\documentclass[leqno,12pt]{amsart}
\usepackage{a4,latexsym,amssymb,amsfonts}

\usepackage{amsmath,amscd} 

\usepackage{enumerate}
\usepackage{amsthm}

\numberwithin{equation}{section}

\newtheorem{theorem}{Theorem}[section]
\newtheorem{lemma}[theorem]{Lemma}
\newtheorem{proposition}[theorem]{Proposition}
\newtheorem{corollary}[theorem]{Corollary}

\theoremstyle{definition}

\newcommand\somma{(\ell+\ell')}
\newcommand\ql{{q_{\ell,\ell'}}}
\newcommand\Ql{{Q_{\ell,\ell'}}}
\newcommand\dll{d_{\ell,\ell'}}
\newcommand\pll{\pi_{\ell,\ell'}} 
\newcommand\rnuno{R_{1}^{\nu}}
\newcommand\rndue{R_{2}^{\nu}}
\newcommand\hll{\cH^{\ell,\ell'}}

\newcommand\hlo{\cH^{\ell,0}}
\newcommand\zll{Z_{\ell,\ell'}}

\newcommand\RR{{\mathbb{R}}}
\newcommand\CC{{\mathbb{C}}}

\newcommand\ZZ{{\mathbb{Z}}}

\newcommand\cL{{\mathcal{L}}}
\newcommand\cH{{\mathcal{H}}}

\newcommand\cE{{\mathcal{E}}}

\newcommand\cC{{\mathcal{C}}}
\newcommand\cR{{\mathcal{R}}}

\newcommand\unme{\frac{1}{2}}

\newcommand\unpi{\frac{1}{p}}

\newcommand\eps{{{\varepsilon}}}

\newcommand\sfera{{S}^{2n-1}}

\newcommand\fdrnu{{\varphi}^{\delta}_{R,\nu}}
\newcommand\sdr{{S}^{\delta}_{R}}

\newcommand\sdrnu{{S}^{\delta}_{R,\nu}}
\newcommand\sdrzero{{S}^{\delta}_{R,0}}
\newcommand\nucleosdrnu{{s}^{\delta}_{R,\nu}}

\newcommand\autovKohn{{\lambda}_{\ell,\ell'}}

\newcommand\resto{\cE_R^\delta}
\newcommand\normaduesfera{L^2{(S^{2n-1})}}

\newcommand\hnur{{h}_{r,\nu}}
\newcommand\hhnur{{\hat{h}}_{r,\nu}}

\newcommand\radcL{\sqrt{\cL}}
\newcommand\radR{\sqrt{R}}

\newcommand\radpi{\sqrt{2\pi}}

\def\d{d}
\def\B{B}

\begin{document}

\title[Riesz means on complex spheres]{$L^p$-summability of Riesz
means for the  sublaplacian on complex
spheres}
\author{Valentina Casarino and Marco M. Peloso}
\address{Dipartimento di Matematica\\
Politecnico di Torino\\
Corso Duca degli Abruzzi 24\\10129 Torino}
\address{Dipartimento di Matematica\\
Universit\`a degli Studi di Milano\\
Via C. Saldini 50\\
20133 Milano}
\email{valentina.casarino@polito.it, marco.peloso@mat.unimi.it}
\thanks{}
\date{\today}

\maketitle

\begin{abstract}
In this paper we study the $L^p$-convergence of the
Riesz means $S^\delta_R(f)$ 
for the sublaplacian on the sphere $S^{2n-1}$ in the complex
$n$-dimensional space $\CC^n$.

We show that $S^\delta_R(f)$ converges to $f$ in $L^p(S^{2n-1})$ when
$\delta>\delta(p):=(2n-1)|\frac12-\frac1p|$.  The index
$\delta(p)$ 
improves the one
found by Alexopoulos and Lohou\'e \cite{AlexLo},
$2n|\frac12-\frac1p|$, and it 
coincides with the one found by Mauceri
\cite{Mauceri}
and, with different methods, by M\"uller \cite{Mueller} in the case of
sublaplacian on the Heisenberg group. \bigskip
\end{abstract}

\section{Introduction and statement of the main result}\medskip

In this paper we study the convergence in the $L^p$-norm
of the Riesz
 means associated with the sublaplacian on complex
spheres.  

Let $S^{2n-1}$  be the unit sphere in $\CC^n$ and let 
$\cL$ be the differential operator on $S^{2n-1}$ 
defined by
\begin{equation}\label{sublap}
\cL:= -\sum_{j<k}M_{jk}\overline{M}_{jk} + \overline{M}_{jk}M_{jk}\, ,
\end{equation}
where
$M_{jk}:=\overline{z}_j
\partial_{z_k}-\overline{z}_k \partial_{z_j}$.
The operator 
$\cL$ is called
the sublaplacian on $S^{2n-1}$.

The starting point in our approach
is the orthogonal  decomposition
of the space of square-integrable functions on the unit sphere
$S^{2n-1}$  in $\CC^n$
as
\begin{equation}\label{decomposition}
L^2 \left( S^{2n-1}\right)= 
\bigoplus_{\ell,\ell'=0}^{+\infty}  \hll\,,
\end{equation}
where
$\hll$ is the space of
 complex spherical harmonics of bidegree
$(\ell,\ell')$, $\ell,\ell'\ge 0$
[ViK, Ch.11].
Each subspace $\hll$ turns out to be  an eigenspace
for $\cL$ with  eigenvalue
$\lambda_{\ell,\ell'}= 2\ell\ell'+(n-1)\somma$ \cite{Geller}.

We denote by $\pll$ the
spectral projection from
$L^{2}(S^{2n-1})$
onto $\hll$.
Then, a function
$f\in L^2 (\sfera)$ may be written as
\begin{equation}\label{decompositionf}
f=\sum_{\ell,\ell'=0}^{+\infty}\pll f\,,
\end{equation}
where the partial sums of the series converge in the $L^2$-topology.
However, if $f\in L^p (\sfera)$, with $p\neq 2$,
then the partial sums 
in general fail to converge to $f$  in the
$L^p$-topology, 
unless  suitable smoothing and convergence factors are inserted.
A classical summability method  is provided by the
Riesz means 
(see \cite{Riesz,Bochner}, 
and also
\cite{DavisChang,Stein}).
 
The Riesz means of index
$\delta\ge 0$ associated with $\cL$ of a function
$f\in \cC^\infty(\sfera)$,
are defined as
\begin{equation}\label{medieBochner1}
S^{\delta}_{R} f
:= \sum_{\ell,\ell'=0}^{+\infty}
\biggl( 1-\frac{\lambda_{\ell,\ell'}}{R}\biggr)^{\delta}_{+} \pll f
\, .
\end{equation}

The main question is to establish the critical exponent
for the convergence of
$S^{\delta}_{R} f$ to $f$ in the $L^p$-norm, 
as $R\to \infty$.
Obviously,
$S^{\delta}_{R}
f$ converges to $f$ in the $L^2$-norm for all $\delta\ge 0$,
while for $1\le p<2$ we can expect convergence in the
$L^p$-norm only for $\delta$ positive
and sufficiently large.\medskip

There exists a vast literature concerning the convergence of
Riesz means in various settings, and related multiplier
theorems. 
For sake of brevity, 
here we only mention the results 
that are most significant for this work.

In the present context, previous results are due
to Alexopoulos and Lohou\'e \cite{AlexLo}.
They observed that,
by adapting their techniques
for  the study of Riesz means associated to
left-invariant sublaplacians on connected Lie groups of polynomial growth,
it is possible to prove that 
the Riesz means of index $\delta$
associated to the sublaplacian $\cL$ on the
sphere $\sfera$
converge in the $L^p$-norm provided
$\delta> 2n\big| \unpi-\unme\big|$.

In the framework of the Heisenberg group, relevant results are due to
Mauceri
and M\"uller, \cite{Mauceri} and \cite{Mueller}.  In these two
independent papers they studied the question of the $L^p$ convergence
of the Riesz means for the sublaplacian $\cL_H$ on the
Heisenberg group $H_{n-1}$, which is biholomorphic equivalent to the
sphere $S^{2n-1}$ taken away the north pole, via the Cayley
transform.  The operator $\cL_H$ is self-adjoint and admits the
spectral decomposition 
$$
\cL_H = \int_0^{+\infty} \lambda\, dE(\lambda)\, .
$$ 
The Riesz means of index $\alpha$ for a Schwartz function $f$ on
$H_{n-1}$, for $R>0$, are defined 
by 
\begin{equation}\label{heisenberg}
\tilde S_R^\alpha (f) := \int_0^{+\infty}
\biggl(1-\frac{\lambda}{R}\biggr)_+^\alpha
\, dE(\lambda) (f)\, .
\end{equation}

Mauceri and M\"uller, independently and with different methods,
proved that, when $f\in L^p(H_{n-1})$, $\tilde S_R^\alpha (f)$ converges to
$f$ in the 
$L^p$-norm provided
$\alpha> (2n-1)\big| \unpi-\unme\big|$.
We mention in passing that, in view of the restriction theorem proved
in \cite{Mueller-annals}, M\"uller studied the convergence of the
Riesz means in the mixed normed spaces $L^p(\CC^n,L^r(\RR))$.
\medskip

We shall prove the following result.
\begin{theorem}\label{teoremaBochner}
Let $n\ge2$, $\sdr f$ be defined in \eqref{medieBochner1}, $1\le
p<\infty$,
and 
$$
\delta (p):= (2n-1)\bigg| \unpi-\unme\bigg|\, .
$$
Then, 
if $\delta>\delta(p)$,
when $f\in L^p (\sfera)$ we have
$$
\sdr f\to f\text{ in\ } L^p (\sfera)\, .
$$
\end{theorem}

Hence, we improve the results
in \cite{AlexLo}, by showing that the critical index
for the
$L^p$-convergence of the Riesz means for $\cL$
is not larger than 
$(2n-1)\big| \unpi-\unme\big|$.  

It would be of great interest 
 to establish whether the condition 
$\delta>(2n-1)\big| \unpi-\unme\big|$ is also necessary.
As it is well known, the corresponding result for the Laplacian in
$\RR^n$, that is the classical Bochner-Riesz summability, is an
outstanding problem, 
when $n\ge3$.  More precisely, let $\cR^\delta$ denote the
multiplier operator on $L^p(\RR^n)$, $1\le p\le\infty$,  given by
$$
\cR^\delta f(x):=
\int_{|\xi|\le1} (1-|\xi|^2)_+^\delta \hat f(\xi)
e^{2\pi ix\cdot\xi}\, d\xi \, .
$$
Then $\cR^\delta f$ converges to $f$ in $L^p(\RR^n)$  when
$\delta>\max( n\big|\frac1p-\frac12\big|-\frac12,0\bigr)$ and
$\big|
\unpi-\unme\big| \ge\frac{1}{n+1}$.  
The former condition is
also necessary, while it is not known whether
the latter condition could be omitted.  This question is settled
when $n=2$ or when  $\cR^\delta$ is restricted  to some particular
subspace of $\RR^n$,
see \cite{Stein}
Ch. IX, or \cite{DavisChang},
e.g..\medskip

In the context of a compact Riemannian manifold of dimension $n$, 
Sogge proved sharp estimates for the
Riesz means associated to a self-adjoint elliptic differential
operator \cite{Sogge87,Soggelibro}.
More precisely, Sogge proved that  the Riesz means of index $\delta$
for an arbitrary elliptic operator of order $m$
converge in $L^1$
provided $\delta>(n-1)/2$,
and that, when the order $m$ is $2$, 
they converge in $L^p$ provided
$\delta>\max( n\big|\frac1p-\frac12\big|-\frac12,0\bigr)$
and 
$\big| \unpi-\unme\big| \ge\frac{1}{n+1}$.  Also in this case it is known
that the former condition is necessary, cfr. 
\cite{Sogge87}.

More recently, Sikora and Tao \cite{TaoSikora}
studied the $L^p$
convergence of the Bochner-Riesz means
associated to the Laplace-Beltrami operator
in two particular situations:
on the complex sphere
$\sfera$ when restricted to 
the subspace $H^p (\sfera)\subset L^p (\sfera)$
of analytic functions on $\sfera$, and on $\RR^n
=\RR^{n-1}\times\RR$, when restricted to the
space of 
cylindrically symmetric functions on $\RR^{n-1}\times\RR$.  
In both cases they obtained sharp results and were able to remove the
condition
$\big| \unpi-\unme\big| \ge\frac{1}{n+1}$
from the assumptions for the $L^p$ 
convergence of the Bochner-Riesz means.\medskip

Returning to the case of $S^{2n-1}$, it is worth noticing that
$2n-1$ represents the topological  dimension
of the CR sphere, while $2n$ is the so-called {\em homogeneous} dimension.
Thus,  our result may be related to
the spectral multiplier theorems
proved by Hebisch
\cite{Hebisch} and 
M\"uller and Stein
\cite{MuellerStein},
where the critical index is determined by the Euclidean dimension
of the underlying space; see \cite{CowlingSikora}
for an analogous result on $SU(2)$
and for a detailed discussion of the subject.
Moreover, recently  Cowling, Klima and Sikora 
proved a H\"ormander-type spectral multiplier theorem
for the Kohn-Laplacian on spaces of $(p,q)$-forms on
$\sfera$ with critical index 
equal to $(2n-1)/2$ \cite{Cowling}.
On the other hand, the index found in \cite{AlexLo} is determined
by the {\em homogeneous} dimension of $\sfera$.

We point out  that 
our Theorem \ref{teoremaBochner}
implies the cited result by Mauceri and
M\"uller (in the non-mixed norm case $r=p$)
for the Riesz means associated to the sublaplacian on the
Heisenberg group. Indeed, 
this fact 
is a consequence of certain contraction
results \cite{Ricci,RicciRubin,DooleyGupta1,DooleyGupta2}, 
as we shall see in Section \ref{final}.\medskip

The proof of the $L^p$-convergence
of  Riesz means follows the following pattern.
First of all, we break the operator
$S^{\delta}_{R}$ into a  sum of certain pieces $\sdrnu$.
Then we  estimate the $(L^p,L^p)$-norm of each  piece
$\sdrnu$ first on a certain Koranyi ball on $\sfera$.
As Sogge's bounds in the Riemannian context,
our estimate relies on a restriction-type lemma,
Lemma \ref{sommeiperboli}.  This lemma in turns follows from sharp
two-parameter estimates for the spectral projections on
the CR sphere recently proved by the first author \cite{Casarino2}.
Secondly, we estimate the
$(L^p,L^p)$-norm of
$\sdrnu$ outside the Koranyi ball.
Here our proof is based on some techniques and ideas,
introduced  by M. Taylor  \cite{Tay2}
and developed by Alexopoulos and
Lohou\'e \cite{AlexLo}. 
In particular, a basic r\^ole in this estimate
is played by the heat kernel associated with the sublaplacian on
$\sfera$.\medskip

Finally,  in this paper we focus the attention on
Riesz means associated to the sublaplacian on $\sfera$.
It would be interesting to treat the case of Riesz means associated to the
Kohn-Laplacian for $(p,q)$-forms as well.  We refer the reader to 
\cite{Geller} for the precise definitions.
However, the analysis of the Kohn-Laplacian, most recently  
developed  in \cite{RicciJeremie} and \cite{Cowling},
seems to be more involved, even in the case of functions;
see also the seminal paper \cite{FollandStein} for the case of the
Heisenberg group.
We hope to extend our methods to cover this more general situation in
the near future.

We shall use  the symbol $C$  to denote
constants which may vary from one formula to the next, and $[x]$ to
denote the greatest integer $\le x$.
\bigskip

\section{Preliminaries}\label{prelim}\medskip

For $n\ge 2$ let $\CC^n$ denote the n-dimensional complex space
endowed with the scalar product
$\langle z,w\rangle:=z_1 \bar w_1 +\cdots+z_n\bar w_n$,  $z,w\in\CC^n$, and
let $S^{2n-1}$ denote  the unit sphere in
$\CC^n$, that is 
$$ 
S^{2n-1} 
:=\bigl\{z=(z_1,\dots,z_n)\in\CC^n:\, \langle z,z\rangle=1 \bigr\}\,.
$$

We recall that, for $z,w\in\sfera$, the formula
\begin{equation}\label{distanza}
\d (z,w):= \left| 1-\langle z,w\rangle \right|^{\unme}\,,
\end{equation}
defines a metric on $S^{2n-1}$, \cite{Rudin}.
The set
 $$
\B (z,r):= \{w\in\sfera:\, \d (z,w)<r \}\,
$$
defines
 the corresponding ball  centered at $z\in\sfera$,
 with radius $r>0$.
It is well known  that there exist two constants $c_1,c_2>0$ such that 
$$
c_1 r^{2n} \le |\B(z,r)|:=\text{Vol}\, (\B(z,r))\le c_2
r^{2n}\,.
$$
Obviously,
$\B (z,r)=\sfera$
when $r>\sqrt{2}$.

Also the operator $\cL$ induces a distance $d_\cL$, called the
control, or Carnot-Caratheodory, distance.  It is well known that 
 $d_\cL$ is
equivalent to the distance $d$ defined in \eqref{distanza}, so we will
use the latter throughout the rest of the paper, see 
e.g. \cite{CD-al}.

\subsubsection*{Spherical harmonics and the spectral projections} 
Consider the space $L^2(S^{2n-1})$,
endowed with the inner product
$$
(f,g):= \int_{S^{2n-1}}
f(\xi) \overline{g(\xi)} \, d\sigma (\xi)\, ,
$$
where $d\sigma$ is the Lebesgue surface measure, which is invariant
under the action of the unitary group $U(n)$.

For $\ell,\ell'\ge 0$
we will denote by  $\hll$ the space of the restrictions to
$S^{2n-1}$  of harmonic polynomials
$p(z,\bar{z})$, homogeneous of degree $\ell$ in
$z$ and of  degree $\ell'$ in $\bar z$.

Then, the subspaces $\hll$ are finite dimensional $d_{\ell,\ell'}$,  $U(n)$-invariant,
mutually orthogonal and their sum is dense in $L^2(S^{2n-1})$, that is 
the decomposition \eqref{decomposition} holds.

A special r\^ole in $\hll$ is played  by the so-called 
{\em zonal}
functions.
Let $\big\{Y_k^{\ell,\ell'}\big\}$, $j=1,\dots,d_{\ell,\ell'}$,
be an orthonormal basis for $\hll$.  
For $(\xi,\eta) \in S^{2n-1}\times S^{2n-1}$ 
set
$$
\zll(\xi,\eta) :=\sum_{k=1}^{d_{\ell,\ell'}}
Y_k^{\ell,\ell'} (\xi)
\overline{Y_k^{\ell,\ell'}(\eta)}
\, .
$$
Then, for all $f\in\hll$ we have
$$
f(\xi)=\int_{S^{2n-1}} \zll(\xi,\eta)f(\eta)\, d\sigma(\eta)\, .
$$ 
Since $\hll$ is finite dimensional, the above pairing makes sense for
all $f\in L^2(\sfera)$.

If we denote by $\pll$ the orthogonal projector onto $\hll$,
for all $f\in L^{2}(S^{2n-1})$, then
\begin{equation}\label{proiettori}
 \pll f= \int_{S^{2n-1}} \zll(\cdot,\eta)f(\eta)\, d\sigma 
\,.
\end{equation}

For each fixed point $\eta\in S^{2n-1}$, the function
$f=\zll(\cdot,\eta)$ is in $\hll$ and it is constant on the orbits of the
stabilizer of $\eta$.

An explicit formula for the zonal function $\zll\in\hll$, for
$\ell\ge\ell'$, is given by
\begin{equation}\label{zonali}
\zll (\xi,\eta)
= \frac{1}{\omega_{2n-1}}
\frac{(\ell+n-2)!}{\ell'! (n-1)!}
    \langle\xi,\eta\rangle^{\ell-\ell'}
P_{\ell'}^{(n-2,\ell-\ell')}(2\langle\xi,\eta\rangle^2-1)\, ,
\end{equation}
where ${\omega_{2n-1}}$ denotes the surface area of ${S^{2n-1}}$,
and $P_{\ell'}^{(n-2,\ell-\ell')}$ is the Jacobi polynomial.
In particular, when $\ell'=0$, 
recall  that $\hlo$ consists of holomorphic polynomials, and
$\cH^{0,\ell}$ consists of polynomials whose complex conjugates are
holomorphic. Since $P_0^{(n-2,\ell-\ell')}\equiv1$,
in this case the zonal function is
$$
Z_{\ell,0} (\xi,\eta)=\frac{1}{\omega_{2n-1}}
\binom{\ell+n-1}{\ell} \langle\xi,\eta\rangle^\ell \, .
$$

For the case $\ell'\ge\ell$, recall that $\zll(\xi,\eta)
=\overline{Z_{\ell',\ell}(\eta,\xi)}$.

Recall that the sublaplacian $\cL$ is defined by \eqref{sublap}.
Then, for $f\in\hll$ we have that
$$
\cL f =
\bigl(2\ell\ell'+(n-1)(\ell+\ell')\bigr)f=:\lambda_{\ell,\ell'}f \, .
$$
Moreover, $\cL$ turns out to be
a self-adjoint positive definite, subelliptic operator.
For these and other properties of $\cL$ we refer the reader to
\cite{Geller}; see also \cite{RicciJeremie}.\medskip

\subsubsection*{Two parameter estimates for the spectral projections}
In our proof we shall use  a sharp two-parameter estimate  the 
spectral projections $\pll$
on complex spheres, recently obtained by the first
author.
\begin{theorem}\label{teoremariassuntivo}
{\rm (\cite{Casarino1,Casarino2})}
Let $n\ge 2$ and let $\ell,\ell'$ be non-negative integers.
Then, for $1\le p\le2$ we have
 \begin{equation}\label{stimaduep}
\|\pll\|_{(L^p,L^2)}\le C
\left( 2\ql+n-1\right)^{\beta(p)} %
\left(1+\Ql\right)^{(n-1) \left(\frac{1}{p}-\frac{1}{2}\right)}\, ,
\end{equation}
where 
$\ql:=\min(\ell,\ell')$, $\Ql:=\max(\ell,\ell')$, and
$$
\beta(p)= 
\begin{cases} (n-1)(\frac{1}{p}-\frac{1}{2}) -\unme \quad&\text{ if }\quad 
1\le p<2 \frac{2n-1}{2n+1}\cr
-\unme\bigl(\frac1p-\frac12\bigr)& \text{ if }\quad 
 2\frac{2n-1}{2n+1}
\le p\le 2.
\end{cases}
$$
\end{theorem}

We remark that, on $S^{2n-1}$, besides the sublaplacian $\cL$,
one may  also consider the Laplace-Beltrami operator
$-\Delta_{S^{2n-1}}$. 
It is possible to work out a joint spectral theory. In particular,
$\hll$ turns out to be a joint  eigenspace  with
eigenvalue $\mu_{\ell,\ell'}$ for $-\Delta_{S^{2n-1}}$, where
$\mu_{\ell,\ell'}:= \somma\left( \ell+\ell'+2n-2\right)$, \cite{Geller}.
Then $\ql$ and $\Ql$ are related to the
 eigenvalues $\lambda_{\ell,\ell'}$ and  $\mu_{\ell,\ell'}$, since
 they  grow, respectively, as 
$\lambda_{\ell,\ell'}/(\ell+\ell')$ and
${\mu_{\ell,\ell'}}^{\unme}$.  This was the spirit
of the results in \cite{Casarino1,Casarino2}.\medskip

\subsubsection*{The heat kernel associated with the sublaplacian}
The proof of Theorem \ref{teoremaBochner} relies on some good estimates
for the heat kernel of  $\cL$.  We shall denote by
$p_{t} (z,w)$ the {\it{heat kernel}} associated with $\cL$,
that is the Schwartz kernel
of the operator $e^{-t\cL}$, for $t>0$.

We shall use a  Gaussian upper estimate,
that is  an upper estimate for the heat kernel, where the space-time scaling
$\d_\cL (z,w)^2/t$ is of Gaussian type, proved in
the real subelliptic framework by
Varopoulos \cite{Var1,Var2}.
In particular he showed 
that there exist two constants $C_1,C_2>0$ such that
\begin{equation}\label{gaussian}
p_{t} (z,w) \le
\frac{C_2}{\big|B(z,\sqrt{t}) \big|}
\exp\{-C_1 \d (z,w)^2/t\} 
\end{equation}
for every $t>0$ and $z,w\in\sfera$.
(The distance appearing in this estimate is in fact the control
distance defined by $\cL$.  We have pointed out that this distance is
equivalent to the distance $\d$ defined in \eqref{distanza}.)
This estimate implies that 
\begin{equation}\label{normapt}
\|p_{t}(z,\cdot)\|_{L^2 (\sfera)}\le
C \big| B(z,{\sqrt{t}})\big|^{-\unme}\,.
\end{equation}

A central idea in our approach is due to M. Taylor \cite{Tay1,Tay2}.
If $m$ is a continuous  function
on $\RR_+$, we may define the operator
$$
m(\cL):= \sum_{\ell,\ell'=0}^{\infty}
m(\lambda_{\ell,\ell'}) \pll\, ,
$$
which is bounded on $L^2 (\sfera)$.
Assume now that $m\in\cC^{\infty}_0$ and $m(\cL)=f(\radcL)$,
with $f$ an even function. This implies
\begin{equation}\label{ideaTaylor}
f(\radcL)= 
\frac{1}{\sqrt{2\pi}}
\int_{-\infty}^{+\infty} \hat{f} (t) \,\cos t\sqrt{\cL}\, dt\, .
\end{equation}
To treat $f(\radcL)$ we use both the fact that $\cos t\radcL$ is bounded on
$L^2 (\sfera)$ and the
 {\it finite propagation speed}
property proved, for real subelliptic
operators, by Melrose \cite{Melrose}.

More precisely, if we denote  by  $G_{t}(z,w)$ the kernel
of the operator $\cos t\radcL$, which 
is a fundamental solution
for the wave equation, and therefore
\begin{equation}\label{propagazione}
\text{supp\,}
(G_t)\subseteq \bigl\{
(z,w)\in\sfera\times\sfera\,:\, \d (z,w)\le |t|
\bigr\}\,. \bigskip
\end{equation}

\section{Proof of Theorem \ref{teoremaBochner}} \medskip

By duality, and to avoid the trivial case $p=2$,
 we may restrict ourselves to the case $1\le p<2$, and we
shall do so in the remainder of the paper, unless explicitely stated.

In order to prove
the convergence of
$\sdr f$ to $f$ as $R\to\infty$
in $L^p$-norm 
it suffices to prove the $L^p$ uniform
boundedness of $\sdr$,  that is 
that, if $1\le p<2$ and $\delta>\delta(p)$, then
$$
\|\sdr f\|_p \le C\|f\|_p
$$
for some constant $C>0$ independent of $R>0$, for all $f\in \cC^\infty(S^{2n-1})$.

Let $\varphi$ be a $\cC_0^{\infty}(\RR)$, with support contained in
$(\unme,2)$
and such that
$\sum_{\nu=-\infty}^{\infty} \varphi (2^{\nu}t)=1$ for every $t>0$,
and set $\varphi_0(t) =1-\sum_{\nu=1}^{\infty} \varphi (2^{\nu}t)$.
Then, for $\nu=1,2,\dots$, and $0\le t\le R$ we set
\begin{equation}\label{fdrnu}
\fdrnu (t):= \biggl(1-\frac{t}{R}\biggr)^\delta_+
\varphi \bigl(
2^{\nu}(1-t/R)\bigr) \, .
\end{equation}
Clearly we have:
\begin{itemize}
\item[(i)]
$\text{supp\,} \fdrnu \subseteq
\bigl( R(1-2^{-\nu+1}),R(1-2^{-\nu-1}) \bigr)$;
\item[(ii)]
$\sup_{t\in\RR} |\fdrnu (t)|\le C\, 2^{-\nu\delta}$;
\item[(iii)] for every $N=1,2,\dots$, there exists $C_N>0$ such that
$$
\big| {\partial_t}^N 
\fdrnu (t)\big|\le C_N
\left( \frac{2^\nu}{R} \right)^N\, .
$$
\end{itemize}

For $\nu=1,2\dots$, we set
\begin{equation}\label{sdrnu}
\sdrnu f:= \sum_{\ell,\ell'=0}^{+\infty}
\fdrnu \left(\autovKohn\right)
\pll f\, .
\end{equation}
Then
\begin{align*}
\sdr f
& = \sum_{\ell,\ell'=0}^{+\infty}
\varphi_0 (1-\autovKohn/R) 
\biggl(1-\frac{\autovKohn}{R}\biggr)^\delta_+ \pll f
\\
&\qquad +
\sum_{\nu=1}^{[\log \radR]}
\sum_{\ell,\ell'=0}^{+\infty}
\fdrnu \left(\autovKohn\right)
\pll f
+ \sum_{\ell,\ell'=0}^{+\infty} \sum_{\nu=[\log \radR]}^{+\infty}
\fdrnu \left(\autovKohn\right) \pll f
\\ 
&=:
\sdrzero f + \sum_{\nu=1}^{[\log \radR]}
\sdrnu  f + \resto f \, .
\end{align*}

It is clear that the main term is the second one, that is 
$\sum_{\nu=1}^{[\log \radR]} \sdrnu f$.
Arguing as in \cite{Sogge87}, it suffices to prove that
there exist constants $C>0$ and $\eps>0$ such that
\begin{equation}\label{stima-da-provare}
\|\sdrnu f \|_{p}\le
C 2^{-\eps \nu}
\| f\|_{p}\,,\qquad\qquad\nu=1,2,\ldots,[\log \sqrt{R}]\, .
\end{equation}

The bound  for $\resto f$
is proved in Lemma 
\eqref{lemmaresto}, while
in order to dispose of the term $\sdrzero f$ we will prove the following.
\begin{proposition}\label{propoKara}
The operator
$\sdrzero$
is bounded on $L^p (\sfera)$, uniformly
in $R>0$, that is, there exists a costant $C>0$ such
that
\begin{equation}\label{stimasdrzero}
\|\sdrzero f\|_p \le C \|f\|_p
\end{equation}
for all $R>0$
and for every $p\in [1, \infty)$.
\end{proposition}

This proposition could probably be proved by an argument of general
nature, 
since the operator is defined by a spectral multiplier that is $C^\infty$
with compact support, for which we would need estimates uniform in
$R$.  However, 
we prefer to prove it in a more direct  way, 
by adapting our techniques for the estimate of
the main term  $\sum_{\nu=1}^{[\log \radR]} \sdrnu f$.
Hence, we defer the proof
to the end of the section.\medskip

Our proof hinges on the following restriction-type lemma.
\begin{lemma}\label{sommeiperboli}
If $1\le p\le 2$, there exists a constant $C>0$
such that 
\begin{equation}\label{rieszthorin}
\Big\|\sum_{\autovKohn\in(R(1-2^{1-\nu}), R(1-2^{-1-\nu}))}
\pll \Big\|_{(L^p, L^2)}
\le C \Bigl(
R^n \max\bigl(2^{-\nu},1/\sqrt{R}\bigr) \Bigr)^{\unpi-\unme} 
\, .
\end{equation}
\end{lemma}
\begin{proof}
Set 
${R_{1}^{\nu}}:=
R(1-2^{1-\nu})$
and
${R_{2}^{\nu}}:= R(1-2^{-1-\nu})$.

We begin by proving that
\begin{equation}\label{sommaiperboleab}
\sum_{{\autovKohn\in(\rnuno,\rndue)}}
(\ell+\ell') \le C
R^2 \max\bigl(2^{-\nu},1/\sqrt{R}\bigr)\, .
\end{equation}
First, we observe that
$\autovKohn\in(\rnuno,\rndue)$ if and only if 
$$
2\rnuno+(n-1)^2 \le (2\ell+n-1)(2\ell'+n-1)\le
2\rndue+(n-1)^2\,.
$$ 
For notational convenience, we write $\tilde\rnuno=2\rnuno+(n-1)^2$ and
$\tilde\rndue=2\rndue+(n-1)^2$.
Now, setting $2\ell+n-1=m$ and
$2\ell'+n-1=m'$, we may write
$$
\sum_{{\autovKohn\in(\rnuno\,,\rndue)}} (\ell+\ell')
=\sum_{\tilde\rnuno \le mm'\le \tilde\rndue} (m-n+1)
\le C \sum_{\tilde\rnuno \le mm'\le \tilde\rndue} m
\, .
$$
Using the elementary bound $\sum_{\{k:\, 0\le a\le k\le b\}}k\le
b(b-a+2)$, for $a,b\in\RR_+$, we obtain that
\begin{align*}
\sum_{{\autovKohn\in(\rnuno\,,\rndue)}} (\ell+\ell')
& \le C\sum_{m'=1}^{[\tilde\rndue]+1}   
\sum_{\{ m:\, \tilde\rnuno\le mm'\le\tilde\rndue\}} m \\
&\le 
C \sum_{m'=1}^{[\tilde\rndue]+1} \frac{\tilde\rndue}{m'}\biggl(
\frac{\tilde\rndue-\tilde\rnuno}{m'}+2\biggr) \\
& \le C \sum_{m'=1}^{[\tilde\rndue]+1} \frac{R}{m'}\biggl(
\frac{R2^{-\nu}}{m'}+2\biggr) \\
& \le C \bigl(R^2 2^{-\nu}+ R\log R\bigr) \\
& \le C 
R^2 \max\bigl(2^{-\nu},1/\sqrt{R}\bigr)\, .
\end{align*}

Using orthogonality,
Theorem \ref{teoremariassuntivo}
and \eqref{sommaiperboleab}
we obtain
\begin{align*}
&\Big\|\sum_{\autovKohn\in(\rnuno,\rndue)}
\pll f \Big\|^2_{\normaduesfera}
= \sum_{\autovKohn\in(\rnuno,\rndue)}
\|\pll f\|^2_{\normaduesfera}\\
&\le C\,
\sum_{\autovKohn\in(\rnuno,\rndue)}
\left( \frac{2\ell\ell'+(n-1)(\ell+\ell')}{\ell+\ell'}\right)^{n-2}
\left(\ell+\ell' \right)^{n-1}
\|f\|^2_{L^1(\sfera)}\\
&\le C R^{n-2}
\sum_{\autovKohn\in(\rnuno,\rndue)}
\left(\ell+\ell' \right) 
\|f\|^2_{L^1(\sfera)} \\
& \le C R^n \max\bigl(2^{-\nu},1/\sqrt{R}\bigr)
\|f\|^2_{L^1(\sfera)}\, .
\end{align*}
Therefore,
$$
\Big\|\sum_{\autovKohn\in(\rnuno,\rndue)} \pll \Big\|_{(L^1,L^2)}
\le C \Bigl( R^n \max\bigl(2^{-\nu},1/\sqrt{R}\bigr)\Bigr)^{\unme}
\, ,
$$
and then a standard application of Riesz-Thorin Theorem finally yields 
the lemma.
\end{proof}

Before proving
($\ref{stima-da-provare}$), we 
provide a uniform bound for the remainder $\resto$.
\begin{proposition}\label{lemmaresto}
There exists a costant $C>0$ such that for all  $1\le p\le 2$  
 \begin{equation}\label{stimarestov}
\|\resto f\|_{L^p}
 \le C\|f\|_{L^p}\, .
\end{equation}
\end{proposition}
\begin{proof}
As in Lemma \ref{sommeiperboli}
we denote 
$R(1-2^{1-\nu})$ and
$R(1-2^{-1-\nu})$ by, resp.,
${R_{1}^{\nu}}$ and ${R_{2}^{\nu}}$.
\par
H\"older's inequality and Lemma
$\ref{sommeiperboli}$
yield
 \begin{align*}
\|\resto f\|_{L^p}^2&\le C\|\resto f\|_{L^2}^2\\ 
& = C
\sum_{\nu=[ \log \radR]}^{+\infty}
  \sum_{\ell,\ell'=0}^{+\infty} \Big\|
\fdrnu \left(\autovKohn\right) \pll f \Big\|_{L^2}^2 
\\
&\le C\sum_{\nu=[ \log \radR]}^{+\infty}
 2^{-2\delta \nu} 
\sum_{\autovKohn\in(\rnuno,\rndue)}\|
\pll f\|_{L^2}^2 
\\
& \le C 
\sum_{\nu=[ \log \radR]}^{+\infty}
 2^{-2\delta \nu} 
R^{(2n-  1)(\unpi-\unme) } 
\|f\|_{L^p}^2\\
& \le C  
 R^{- \delta  } 
R^{(2n-1)(\unpi-\unme)}
\|f\|_{L^p}^2\\
& \le C  
 R^{-\big( \delta -
{(2n-1)(\unpi-\unme)}\big)}
\|f\|_{L^p}^2\\
&\le C
\|f\|_{L^p}^2\,,
\end{align*}
since $\delta>\delta(p)$,
where
$\delta (p)=
(2n-1)\big| \unpi-\unme\big|$
is as in Theorem \ref{teoremaBochner}, thus
proving \eqref{stimarestov}.
\end{proof}

Now we  turn back to the proof of  \eqref{stima-da-provare}.  
First, by using Theorem \ref{teoremariassuntivo} and 
the restriction-type estimate \eqref{rieszthorin}
we   prove an $(L^p,L^2)$-bound for $\sdrnu$.

\begin{proposition}\label{lemmasdrnu}
For $\nu=1,2,\dots,[\log \sqrt{R}]$  and $1\le p\le 2$  we have
 \begin{equation}\label{stimasdrnu}
\|\sdrnu f\|_{L^2(\sfera)}\le
C 2^{-\delta \nu}
2^{-\nu(\unpi-\unme)}
\sqrt{R}^{2n(\frac{1}{p}-\frac{1}{2})}
 \|f\|_{L^p(\sfera)}\, .
\end{equation}
\end{proposition}
\begin{proof}
As a consequence of orthogonality we have
\begin{align*}
\|\sdrnu f\|^2_{\normaduesfera}
&=\sum_{\ell,\ell'=0}^{+\infty}
\fdrnu (\autovKohn)^2
\|\pll f\|^2_{\normaduesfera}\\
&= \sum_{\autovKohn\in(\rnuno,\rndue)}
\fdrnu (\autovKohn)^2
\|\pll f\|^2_{\normaduesfera}\\
&\le C\,2^{-2\nu\delta}
\sum_{\autovKohn\in(\rnuno,\rndue)}
\|\pll f\|^2_{\normaduesfera} \\
&= 
C\,2^{-2\nu\delta}
\Big\|\sum_{\autovKohn \in(\rnuno,\rndue)}
\pll f\Big\|^2_{\normaduesfera}\\
&\le C\,2^{-2\nu\delta} R^{2n(\unpi-\unme)}
2^{-2\nu(\unpi-\unme)}\|f\|^2_{L^p (\sfera)} \, ,
\end{align*}
for all $1\le p\le 2$, 
where we used the fact that
$\sup_{t\in\RR}|\fdrnu (t)|\le C\,2^{-\delta\nu}$
and Lemma \ref{sommeiperboli}.  This gives the result.
\end{proof}

Consider now a Koranyi ball $\B$ in $\sfera$ with radius
$\frac{2^{\nu}}{\sqrt{R}}$.
 H\"older's inequality yields
\begin{equation}\label{Holder}
\|\sdrnu f\|_{L^p (\B)}
\le |\B|^{\unpi-\unme} \|\sdrnu f\|_{L^2 (\B)}\, .
\end{equation}

By \eqref{Holder} and \eqref{stimasdrnu}
it follows that
\begin{align}
\|\sdrnu f\|_{L^p (\B)}
&\le C
{\left(\frac{2^{\nu}}{\sqrt{R}}\right)}^{2n(\unpi-\unme)}
\|\sdrnu f\|_{L^2 (\B)}\notag\\
&\le C
{\left(\frac{2^{\nu}}{\sqrt{R}}\right)}^{2n(\unpi-\unme)}
 2^{-\nu\delta}
2^{-\nu(\unpi-\unme)}
\sqrt{R}^{2n(\unpi-\unme)}  \|f\|_{L^p(\sfera)} \notag \\
&\le C
2^{-\nu (\delta-\delta(p))} \| f\|_{L^p (\sfera)}\,. \label{B-K}
\end{align}

It only remains to prove an analogous estimate  for
$\|\sdrnu f\|_{L^p (\sfera\setminus \B)}$.  It suffices to show
that for every $\gamma >0$ there exists $\eps_0>0$ such that
\begin{equation}\label{stimafuori}
\int_{\{ w:\, \sqrt{R}\d (z,w)> 2^{\nu(1+\gamma)}\} }
\big|
\nucleosdrnu (z,w)\big|\,d\sigma (w)
\le C2^{-{\nu\eps_0}}\,,
\end{equation}
where $\nucleosdrnu$ denotes the kernel of  $\sdrnu$.
Indeed, \eqref{B-K} holds true, with the same proof, when $\B$ 
has radius $\frac{2^{\nu(1+\gamma)}}{\sqrt{R}}$ and $\gamma>0$
is sufficiently small.
Moreover, 
by Schur's test,
\eqref{stimafuori} will imply that $\sdrnu$ is bounded on
$L^p(\sfera\setminus \B)$ with norm less than or equal to $ C2^{-{\nu\eps_0}}$.\medskip

Set $r=\sqrt{R}$
and define
\begin{equation}\label{hnur}
\hnur (\lambda):=
\bigl(
1-\lambda^2/r^2 \bigr)^\delta_+
e^{\lambda^2/r^2}
\varphi \bigl( 2^\nu\bigl(1-\lambda^2/r^2 \bigr) \bigr)\ .
\end{equation}
where $\varphi$ is the
function defined at the beginning of
the section.

\begin{lemma}\label{proprietahnur}
The functions $\hnur$, defined by
\eqref{hnur},
satisfy the following properties:
\begin{itemize}
\item[(i)]
$\left| {\rm supp\,}
\hnur\right|\le
C\,r2^{-\nu}$;
\item[(ii)] for every non-negative integer $k$
there exists a constant $C_k$ such that for all $s>0$
$$
\int_{|t|\ge s} |\hhnur (t)|\,dt \le c_k s^{-k}
r^{-k} 2^{(k-\delta)\nu}\, ;
$$
\item[(iii)]
$\big\| \hnur (\radcL) \big\|_{(L^2,L^2)} \le C 2^{-\delta\nu}$.
\end{itemize}
\end{lemma}

\begin{proof}
Property (i) is clear, (iii) follows at once since 
$\big\| \hnur (\radcL) \big\|_{(L^2,L^2)} \le\|\hnur\|_{\infty}\le 
C 2^{-\delta\nu}$. 
It is easy to check that for all $k$ there exists $C_k >0$ such that
$$
\big\| \hnur ^{(k)} \big\|_{L^{\infty}} \le
C_k
r^{-k}2^{(k-\delta)\nu}\,.
$$
Then integrating by parts $k+1$ times we obtain
\begin{align*}
\int_{|t|\ge s} \left| \hhnur (t)\right|\,dt 
&\le 
\int_{|t|\ge s}
\int \left|
\frac{\partial^{k+1}}{\partial w ^{k+1}}
\hnur(w)\right|dw\,
\frac{1}{|t|^{k+1}}\, dt  \\
&\le C  s^{-k}
r^{-k}2^{-(\delta-k)\nu}\,.
\end{align*}
This proves (ii).
\end{proof}

Observe that 
if 
$\fdrnu$ is the function defined by
\eqref{fdrnu}, then we 
may  write
$$
\fdrnu (\cL) = \hnur (\radcL) e^{-\cL/r^2}\,.
$$

Clearly
$$
\nucleosdrnu (z,w)= \hnur (\radcL) p_{1/r^2}(z,w)\,,
$$
since the operators $\hnur (\radcL)$ and $e^{-\frac{\cL}{r^2}}$
are self-adjoint and commute. Here the operator
$\hnur (\radcL)$ is acting in the variable $w$. \medskip

In order to prove \eqref{stimafuori}, we call 
the integration set 
$$
A(z) := \{ w:\, r\d (z,w)> 2^{\nu(1+\gamma)}\} 
$$
and decompose it
in annuli in the following way.
Let $i\in\ZZ^+$ be such that
$$
2^{i-1}< r\le 2^i\,.
$$
and consider the shells
$$
A_j (z):= \{w\,:\, 2^j\le \d (z,w)\le 2^{j+1}\}\,.
$$
Then
$$
A(z) \subseteq \bigcup_{\nu(1+\gamma)-i\le j\le0} A_j(z)\, .
$$

Next we write
\begin{align*}
\nucleosdrnu (z,\zeta)
&=  \Bigl( \hnur (\radcL)\, 
p_{1/r^2}(z,\cdot) \Bigr) (\zeta) \\
&={\frac{1}{\radpi}}
 \int_{-\infty}^{\infty}
{\hhnur} (t) \Bigl( \cos t\radcL\,
\,p_{1/r^2}(z,\cdot) \Bigr) (\zeta)\,dt \\
&={\frac{1}{\radpi}}
\int_{-\infty}^{\infty}
{\hhnur}(t)
\Bigl( \cos t\radcL\,\bigl( 
\,p_{1/r^2}(z,\cdot)
\chi_{\{ w:\,\d (z,w)\le{2^{j-1}} \} } \\
&\qquad\qquad\qquad + p_{1/r^2}(z,\cdot)
\chi_{\{w:\, \d (z,w)> {2^{j-1}}\} } \bigl)\Bigl)(\zeta)\,dt \\
& =:{\rm I}_{\nu,r}(z,\zeta)+ {\rm II}_{\nu,r}(z,\zeta)\,.
\end{align*}

Notice that, when $\zeta\in A_j (z)$, so that $\d(z,\zeta)\ge 2^j$,
$$
{\rm I}_{\nu,r}(z,\zeta)
= \frac{1}{\radpi}
\int_{|t|\ge 2^{j-1}} {\hhnur}(t)
\cos t\radcL\,\Bigl(p_{1/r^2}(z,\cdot)
\chi_{\{ w:\,\d (z,w)\le{2^{j-1}} \} }\Bigr)(\zeta)\, dt\, ,
$$
using the finite speed propagation property
and the fact that we are integrating on the set
$\{ w:\,\d (z,w)\le{2^{j-1}}\}$.

Now, using H\"older inequality, the fact that 
$\cos t\radcL$ is a bounded operator on $L^2 (\sfera)$
and \eqref{normapt},
we obtain that
\begin{align*}
\|{\rm I}_{\nu,r}(z,\cdot)\|_{L^1 (A_j (z))}
&\le |A_j (z)|^{\unme} \frac{1}{\radpi}
\int_{|t|\ge 2^{j-1}} 
\big|{\hhnur}(t)\big|\, \|p_{1/r^2}(z,\cdot) \|_{L^2} \,dt \\ 
&\le C\,\frac{ |A_j (z)|^{\unme}}{ 
|B( z,\sqrt2/r)|^{\unme}}
\int_{|t|\ge 2^{j-1}} |{\hhnur}(t)|\,dt \\
&\le C\,
\frac{|B(z, 2^j)|^{\unme}}{|B(z,{2^{-i}})|^{\unme}}
\int_{|t|\ge 2^{j-1}} \big|{\hhnur}(t)\big|\,dt \\
&\le C\, { 2^{ n(i+j)}} \int_{|t|\ge 2^{j-1}}
\big|{\hhnur}(t)\big| \,dt\,. 
\end{align*}
Applying (ii) in Lemma \ref{proprietahnur}
we see that for every positive integer $k$ there exists $C_k$ such that
\begin{align*}
\|{\rm I}_{\nu,r}(z,\cdot)\|_{L^1 (A_j (z))}
&\le
C_k { 2^{ n(i+j)}} 2^{ -jk} r^{ -k} 2^{ k} 2^{ (k-\delta)\nu} \\
&\le 
C_k { 2^{ j(n-k)}} r^{ n-k} 2^{ k} 2^{ (k-\delta)\nu}\, .
\end{align*} 
By choosing $k>n(\gamma+1)/\gamma$, we then obtain
\begin{align*}
\sum_{\nu(1+\gamma)-i\le j\le0}
\|{\rm I}_{\nu,r}(z,\cdot)\|_{L^1 (A_j (z))}
&\le C_k  \sum_{\nu(1+\gamma)-i\le j\le0}
2^{j(n-k)} r^{ n-k} 2^{(k-\delta)\nu} \\
&\le C_k  2^{(i-\nu(1+\gamma))(k-n)} r^{n-k}
2^{ (k-\delta)\nu} \\
&\le C_k  2^{-\nu(1+\gamma)(k-n)} 2^{ (k-\delta)\nu}\\
&\le C  2^{ -\nu \eps_1 } \,,
\end{align*}
where
$\eps_{1} := \delta-k +(1+\gamma)(k-n)$ is strictly positive.

Next, using
\eqref{ideaTaylor}, \eqref{gaussian} and (iii) in Lemma
\ref{proprietahnur} we see that
\begin{align*}
 \|{\rm II}_{\nu,r}(z,\cdot)\|_{L^1 (A_j (z))} 
&\le C\,
|A_j (z)|^{\unme} \,
\Big\|\hnur (\radcL) \cos t\radcL \Bigl(p_{1/r^2}(z,\cdot)
\chi_{\{w:\,\d (z,w)> 2^{j-1} \} }\Bigr) \Big\|_{L^2 (\sfera)} \\
&\le C\,
|A_j (z)|^{\unme}
\, \big\|\hnur \big\|_{{\infty}} 
\big\|p_{1/r^2}(z,\cdot) \chi_{\{w:\,\d (z,w)> {2^{j-1}}\}}
\big\|_{L^2 (\sfera)}\\ 
&\le C \frac{|A_j (z)|^{\unme}}{|B(z,1/r)|^{\unme}} 2^{-\delta\nu}
e^{ -C_1 (2^{2j-2}r^2)} \\
&\le C
2^{(j+i)n} 2^{-\delta\nu}e^{-C_1 2^{2j+2i}}\, ,
\end{align*}
for some positive constant $C$.
To conclude our proof, we have only to estimate the sum
\begin{align*}
\sum_{\nu(1+\gamma)-i\le j\le0}
\|{\rm II}_{\nu,r}(z,\cdot)\|_{L^1 (A_j (z))}
&\le C\sum_{\nu(1+\gamma)-i\le j\le0}
2^{(j+i)n} 2^{-\delta\nu} e^{ -C_1 2^{2j+2i}} \\
&\le C 2^{-\delta\nu} \sum_{k\ge\nu(1+\gamma)} 2^{nk}e^{-C_12^{2k}} \\
&\le C 2^{-\delta\nu}
\end{align*}
for every $\nu\ge 1$, for some positive costant $C$.
Finally,
\begin{align*}
& \big\|\nucleosdrnu 
(z,\cdot)\big\|_{L^1( \{ w:\,r \d (z,w)>2^{\nu (1+\gamma)}\} )} \\
&\le \sum_{\nu(1+\gamma)-i\le j\le0}
\|{\rm I}_{\nu,r}(z,\cdot)\|_{L^1 (A_j (z))}
+ \sum_{\nu(1+\gamma)-i\le j\le0} 
\|{\rm II}_{\nu,r}(z,\cdot)\|_{L^1 (A_j (z))}\\
&\le C(2^{-\nu\eps_1}+ 2^{ -\nu \delta}) \\
&\le C 2^{-\nu\eps_0}\,,
\end{align*}
yielding \eqref{stimafuori}. 
\qed

To conclude the proof of
Theorem \ref{teoremaBochner}, it therefore remains
to prove
Proposition \ref{propoKara}.

\proof[Proof of Proposition \ref{propoKara}.]
By reasoning as in  the proof of Lemma
\ref{sommeiperboli}, we may easily show that
for some constant $C>0$ 
$$\sum_{\autovKohn\in(0, R/2)}
(\ell+\ell')
\le C 
R^{2} 
\,,$$
whence
$$\Big\|\sum_{\autovKohn\in(0, R/2)}
\pll \Big\|_{ L^2}^2
\le C 
R^{n-2}
\sum_{\autovKohn\in(0, R/2)}
(\ell+\ell') \le CR^n
\,,
$$
so that the standard  interpolation argument yields
\begin{equation}\label{rieszthorinsdro}
\Big\|\sum_{\autovKohn\in(0, R/2)}
\pll \Big\|_{(L^p, L^2)}
\le C 
R^{n(\unpi-\unme)} 
\, \,\text{for all } 1\le p\le 2\,  .
\end{equation}
Now 
take a Koranyi ball $B_0$ in $\sfera$ with radius $1/\sqrt R$. Then (\ref{Holder})
and (\ref{rieszthorinsdro})
give
\begin{align}
\|\sdrzero f\|_{L^p (\B_0)}&
\le C R^{-n(\unpi-\unme)} \|\sdrzero f\|_{L^2 (\B_0)}\notag\\
&\le C R^{-n(\unpi-\unme)} \|\sdrzero f\|_{L^2 (\sfera)}\notag \\
& \le C\| f\|_{L^p (\sfera)}\,. \label{sdr0}
\end{align}
Now observe that the kernel of the operator
$\sdrzero$ is essentially given by
$
{{s}^{\delta}_{R,0}}(z,w)={{h}_{r,0}} (\radcL) p_{1/r^2}(z,w)$, where
${{h}_{r,0}} $ has been defined in
(\ref{hnur}).
By following the same pattern 
as before, we split the $L^1$-norm of the kernel
${{s}^{\delta}_{R,0}}$
outside the ball $B_0$  into two parts.
Analogous (but easier) computations as before
yield
 \begin{align*}
 \big\|
{{s}^{\delta}_{R,0}} 
(z,\cdot)\big\|_{L^1
(  \{ w:\,\sqrt R  d (z,w)>1\} )}
&\le \sum_{-i\le j\le0}
\|{\rm I}_{0,r}(z,\cdot)\|_{L^1 (A_j (z))}
+ \sum_{-i\le j\le0} 
\|{\rm II}_{0,r}(z,\cdot)\|_{L^1 (A_j (z))}\\
&\le C. \\
\end{align*}
This inequality, combined with
(\ref{sdr0}), 
yields (\ref{stimasdrzero}).
\medskip
\qed

The following Corollary is contained in Theorem
\ref{teoremaBochner}, but we state it for sake of
clarity.
\begin{corollary}\label{sopradeltacritico}
Let $\delta>n-\unme$ and $1\le p<\infty$.  Then for $f\in L^p (\sfera)$
$$
\sdr f\to f
$$
in the $L^p$-norm, as $R\to +\infty$.\bigskip
\end{corollary}

\section{Final remarks}\label{final}\medskip

\subsubsection*{Results on the Heisenberg group}
As already mentioned in the Introduction,
Mauceri \cite{Mauceri} and M\"uller \cite{Mueller} 
proved a summability result for Riesz means for the
sublaplacian $\cL_H$ on the Heisenberg group $H_{n-1}$, defined in
\eqref{heisenberg}.  
They showed that $\tilde S_R^\alpha f$ converges to $f$ in $L^p(H_{n-1})$
when $\alpha>(2n-1)\big|\frac1p-\frac12\big|$.
M\"uller actually proved a more general result concerning the
convergence of $\tilde S_R^\alpha $ in the mixed-normed spaces
$L^p (\CC^{n-2}, L^r (\RR))$.
It would be of some interest to extend the results of \cite{Casarino1,
  Casarino2} and the ones in the present paper to the case of the
mixed normed spaces on the sphere $S^{2n-1}$.

By a contraction argument, it is possible to see that 
Theorem \ref{teoremaBochner} implies 
the  result of Mauceri and M\"uller  on the Heisenberg group 
(in the case $p=r$).
When $n=2$, this can be proved by means of 
the contraction theorem for multiplier operators on $L^p (SU(2))$
to multipliers on $L^p (H_{1})$
proved by Ricci and Rubin (\cite{RicciRubin}, pag.574, 
see  also \cite{CowlingSikora}, Section 5).

When $n>2$, $L^p$-boundedness 
of Riesz multipliers on the Heisenberg group
$H_{n-1}$ follows 
 from a trasference result
   of Dooley and Gupta
(see \cite{DooleyGupta2}, Theorem 4.1).
More precisely, if $G$ is the Lie group $SU(n,1)$, with Iwasawa
decomposition $G=KAN$, where $K=U(n)$, if
 $\overline{N}$ is the image of $N$ under the Cartan
involution 
and $M$ is the centralizer of $A$ in $K$,
Dooley and Gupta studied transference in the context of 
contractions of $K$ to 
$\overline{N} M$ or $K/M$ to 
$\overline{N}$.
When $G=SU (n,1)$, 
$K/M$ is $\sfera$ and 
$\overline{N}$ is the Heisenberg group 
$H_{n-1}$.
By applying their contraction principle to the multiplier
 $(1-|\xi|)^{\delta}_+$, we  obtain
 the result in \cite{Mauceri} and \cite{Mueller}. 
\medskip

\subsubsection*{Necessary conditions}
The contraction principle for Fourier 
multipliers from
$L^p(\sfera)$ to $L^p (H_{n-1})$
proved in \cite{RicciRubin} and \cite{DooleyGupta2}
can be used to derive necessary conditions as well.
In particular, since the Riesz means  on the Heisenberg group
$H_{n-1}$ are not bounded  in $L^p(H_{n-1})$ for
$\delta\le \max \bigl(0,(2n-2)\big|\unpi-\unme\big|-\unme\bigr)$,
then the Riesz  means $\sdr$ are unbounded on $L^p(\sfera)$
for $\delta$ in the same range.

Although our result does not show whether $\delta(p)=
(2n-1)\big|\frac1p-\frac12\big|$ is the critical index, it provides, as observed above,
another proof of Mauceri and M\"uller's result on the Heisenberg
group. This fact may be viewed as an indication that $\delta(p)$ could
actually be the critical index.
\medskip

\subsubsection*{On the Riesz kernel}
It would be of great interest  to obtain an explicit expression
for the Riesz kernel, 
that we denote by $s^\delta_R$,
from which to deduce   
precise pointwise  estimates.
In the case of a self-adjoint, elliptic operator of order
$m$
on a
compact Riemannian manifold $M$ of dimension $n$, 
the Bochner-Riesz kernel $K_R^\delta$  
satisfies the pointwise estimate
$$
K_R^\delta(x,y)\le C R^{n/m}\bigl(1+R^{1/m}d(x,y)\bigr)^{-1-\delta}\,
,
$$
where $x,y\in M$ and $d(x,y)$ is their distance
in any $\cC^\infty$ metric on $M$ \cite{Hormander}.
 
What we know is that, as a
consequence of \eqref{proiettori} 
 \begin{align*}
\sdr f (z)&=
\sum_{\ell,\ell'=0}^{+\infty}
\biggl( 1-\frac{\lambda_{\ell,\ell'}}{R}\biggr)^\delta_+
\int_{\sfera} \zll (z,w) f(w)\,dw 
&=: \int_{\sfera} s^\delta_R (z,w)\,f(w)\,dw\, .
\end{align*}
Then,
the Riesz kernel $s^\delta_R$ is given by
\begin{align*}
s^\delta_R(z,w):
& = \sum_{\ell,\ell'=0}^{+\infty}
\biggl( 1-\frac{\lambda_{\ell,\ell'}}{R}\biggr)^\delta_+
\zll ( z,w) \\
& = \sum_{\ell,\ell'=0}^{+\infty}
\biggl( 1-\frac{\lambda_{\ell,\ell'}}{R}\biggr)^\delta_+
\frac{\dll}{\omega_{2n-1}}
\frac{\ql! (n-2)!}{(\ql+n-2)!}
e^{i(\ell'-\ell)\varphi}
    (\cos\theta)^{|\ell-\ell'|}
P_{\ql}^{(n-2,|\ell-\ell'|)}(\cos 2\theta)\, ,
\end{align*}
where we write $\langle z,w\rangle=\cos \theta e^{i\varphi}$, with
$\theta\in[0,\pi/2]$ and $\varphi\in[0,2\pi]$.
\bigskip

At this point we could follow the approach introduced by Szeg\"o
in (\cite{Szego}, Ch.  IX)
and developed in \cite{BonamiClerc}
to write $s^\delta_R$ as a
Ces\`aro-type kernel 
for expansions in terms of  the disc polynomials
$e^{i\beta\varphi}
    (\cos\theta)^{|\beta|}
P_{k}^{(\alpha,|\beta|)}(\cos 2\theta)\, $,
with $\alpha=n-2$, $\beta=\ell-\ell'$ and $k=\ql$.
Anyway the presence of the complex term, which is not trivial off the diagonal $\ell=\ell'$,
makes the computations much more involved than in the real case.

\end{document}